\documentclass{article}
\usepackage[final]{graphicx}
\usepackage{psfrag}
\usepackage{amsmath,amsfonts,amssymb,amsxtra,subeqnarray,bm}

\newlength{\extramargin}
\newcommand{\Real}{\ensuremath{{\mathbb{R}}}}

\newcommand{\Complex}{\ensuremath{{\mathbb{C}}}}

\newcommand{\V}{\ensuremath{\mathcal V}}
\newcommand{\E}{\ensuremath{\mathbb E}}

\newcommand{\W}{\ensuremath{\mathcal W}}

\newcommand{\B}{\ensuremath{\mathcal B}}

\newcommand{\one}{\ensuremath{{\mathbf{1}}}}

\newtheorem{theorem}{Theorem}

\newtheorem{lemma}{Lemma}

\newtheorem{definition}{Definition}
\newtheorem{remark}{Remark}

\newenvironment{proof}{\noindent {\bf Proof.}}{\hfill \hspace*{1pt}\hfill$\blacksquare$}

\begin{document}
\title{Synchronization of linear oscillators coupled through dynamic networks with interior nodes}
\author{S. Emre Tuna\footnote{The author is with Department of
Electrical and Electronics Engineering, Middle East Technical
University, 06800 Ankara, Turkey. Email: {\tt etuna@metu.edu.tr}}}
\maketitle

\begin{abstract}
Synchronization is studied in an array of identical linear
oscillators of arbitrary order, coupled through a dynamic network
comprising dissipative connectors (e.g., dampers) and restorative
connectors (e.g., springs). The coupling network is allowed to
contain interior nodes, i.e., those that are not directly connected
to an oscillator. It is shown that the oscillators asymptotically
synchronize if and only if the Schur complement (with respect to the
boundary nodes) of the complex-valued Laplacian matrix representing
the coupling has a single eigenvalue on the imaginary axis.
\end{abstract}

\section{Introduction}

Consider the dynamic coupling network with four nodes, shown in
Fig.~\ref{fig:network}; {\em dynamic} because it contains energy
storage components (inductors). This network can be represented by a
pair of Laplacian matrices $(D,\,R)$ where $D$ locates the
dissipative connectors (i.e., the resistor with conductance
$g_{13}$) and $R$ the restorative connectors (i.e., the inductors
with inductances $\ell_{12},\,\ell_{23},\,\ell_{34}$) as
\begin{eqnarray*}
D=\left[\begin{array}{cccc}
g_{13}&0&-g_{13}&0\\
0&0&0&0\\
-g_{13}&0&g_{13}&0\\
0&0&0&0\\
\end{array}\right]\,,\qquad
R=\left[\begin{array}{cccc}
\ell_{12}^{-1}&-\ell_{12}^{-1}&0&0\\
-\ell_{12}^{-1}&\ell_{12}^{-1}+\ell_{23}^{-1}&-\ell_{23}^{-1}&0\\
0&-\ell_{23}^{-1}&\ell_{23}^{-1}+\ell_{34}^{-1}&-\ell_{34}^{-1}\\
0&0&-\ell_{34}^{-1}&\ell_{34}^{-1}
\end{array}\right]\,.
\end{eqnarray*}
Clearly, the pair $(D,\,R)$ can be fused into a single entity
$[D+jR]$, a complex-valued Laplacian, without any loss of
information. It turns out that this fusion has merits beyond mere
notational convenience: The collective behavior of an array of
coupled oscillators is closely related to the spectral properties of
the complex-valued Laplacian representing the network through which
the oscillators are coupled. Let us elaborate on this point.

\begin{figure}[h]
\begin{center}
\includegraphics[scale=0.4]{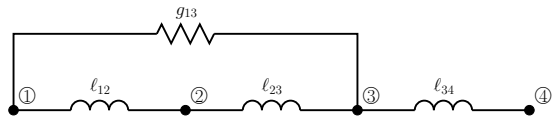}
\caption{Dynamic coupling network with four
nodes.}\label{fig:network}
\end{center}
\end{figure}

Suppose now to the network in Fig.~\ref{fig:network} we connect four
identical harmonic oscillators (LC-tanks) as shown in
Fig.~\ref{fig:networkLC2}. Will these second-order circuits
eventually oscillate in unison? The answer to this question was
given in a recent work \cite{tuna19}, where it was shown that a
coupled array of harmonic oscillators asymptotically synchronize
when (and only when) the coupling matrix $[D+jR]$ has a single
eigenvalue on the imaginary axis. A somewhat interesting feature of
this result is that synchronization (or its absence) is independent
of the characteristic frequency of the harmonic oscillators. This
immediately brings up the question: Does this independence exist as
well for higher-order oscillators, having two or more characteristic
frequencies? which motivates us to conduct the analysis presented in
this paper. To illustrate the setup associated to this question, let
us visit once again the network in Fig.~\ref{fig:network}. Consider
this time the case where we employ our network to couple
fourth-order linear oscillator circuits, each having two
characteristic frequencies. Furthermore, suppose that the oscillator
connected to the third node becomes dysfunctional during operation
and therefore is disconnected from the network (or has never been
connected to the network in the first place). This scenario is
depicted in Fig.~\ref{fig:networkLC4}, where the node
$\raisebox{.5pt}{\textcircled{\raisebox{-.9pt} {3}}}$ is said to be
an {\em interior node}\footnote{This term is borrowed from
\cite{dorfler13}.} of the network, for it is not directly connected
to an oscillator. The simple circuitry shown in
Fig.~\ref{fig:networkLC4} exemplifies the general setup we explore
in this paper: an array of identical linear time-invariant (LTI)
oscillators coupled through an LTI network with interior nodes,
containing not only dissipative components (e.g., resistors,
dampers) but restorative ones (e.g., inductors, springs) as well.
And what we study in this setup is the problem of synchronization.
We state our findings next.

\begin{figure}[h]
\begin{center}
\includegraphics[scale=0.4]{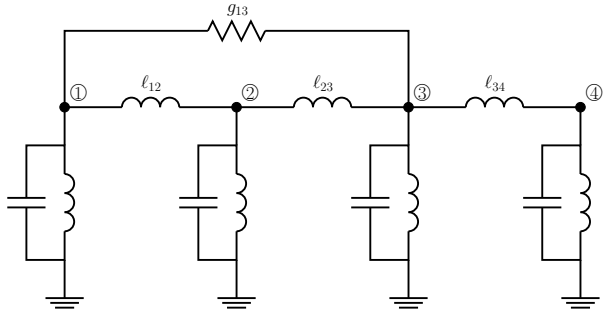}
\caption{Harmonic oscillators coupled through a dynamic
network.}\label{fig:networkLC2}
\end{center}
\end{figure}

For the setup described above (see the next section for the formal
description) we show that what has been established in \cite{tuna19}
for harmonic oscillators is true in general for higher-order
systems. Namely, in the absence of interior nodes, an array of
coupled oscillators asymptotically synchronize if and only if the
complex-valued Laplacian matrix representing the coupling network
has a single eigenvalue on the imaginary axis. For the more general
case, where the network does have interior nodes, this eigenvalue
test remains valid, except for the difference that instead of the
Laplacian itself it has to be applied to the Schur complement of the
Laplacian with respect to the set of non-interior nodes.

\begin{figure}[h]
\begin{center}
\includegraphics[scale=0.4]{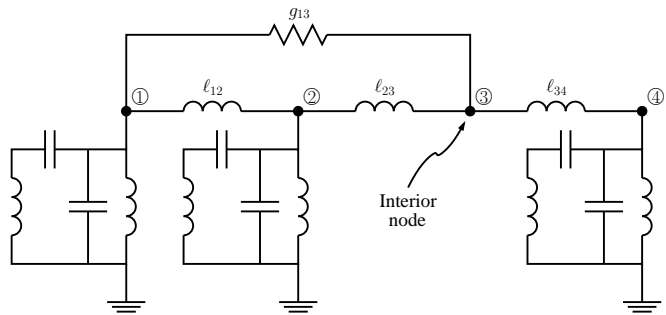}
\caption
{Fourth-order linear oscillators coupled through a network
with an interior node.
%(The node $\raisebox{.5pt}{\textcircled{\raisebox{-.9pt} {3}}}$ is the interior node, for it is not directly connected to an oscillator.)
}\label{fig:networkLC4}
\end{center}
\end{figure}

In order to lay bare the contours of our contribution, we now
attempt to place the above-mentioned main theorem of this paper with
respect to some of the related work in the literature. To this end,
let us impart some details first. The isolated (uncoupled)
oscillator dynamics we study here read \cite[Ch.~11]{lax96}
\begin{eqnarray}\label{eqn:one}
M{\ddot x}_{i}+Kx_{i}=0
\end{eqnarray}
where the constant matrices $M,\,K$ are symmetric positive definite
and $x_{i}$ is a vector. Note that this LTI differential equation is
a natural generalization of the dynamics of an harmonic oscillator,
for which $M$, $K$, and $x_{i}$ are scalar. Even though we have so
far motivated our problem through electrical circuits, the
equation~\eqref{eqn:one} plays an even more important role in
mechanics \cite[Ch.~V]{landau76}. The linearization of a Lagrangian
system about a stable equilibrium enjoys the form~\eqref{eqn:one},
which successfully represents the behavior of the actual system
undergoing small vibrations, e.g., an $n$-link pendulum. This
allows, for instance, the main theorem of this paper to determine
whether an array of mechanical oscillators (undergoing small
vibrations) coupled through a network consisting of springs and
dampers asymptotically synchronize or not; see
Fig.~\ref{fig:pendula3}. To the best of our knowledge, works that
investigate conditions on the coupling network that ensure
synchronization in an array of oscillators with individual
oscillator dynamics~\eqref{eqn:one} are relatively few, save the
special case of harmonic oscillators; see, for instance,
\cite{ren08,su09,zhou12,sun15,tuna17}. A condition for
synchronization relevant to our setup is provided in
\cite[Cor.~14]{tuna19} under the assumptions (i) that the
restorative coupling coefficients are small (e.g., the springs in
Fig.~\ref{fig:pendula3} are weak) and (ii) that the coupling network
contains no interior nodes. In this paper we dispense with both of
these assumptions. To circumvent the difficulty raised by the
interior nodes (which initially pose an obstacle for standard
analysis techniques) we work with the Schur complement
\cite{zhang05} of the complex-valued Laplacian matrix representing
the coupling. This technique (also known as Kron reduction by
engineers) has proven useful in the analysis of electrical networks
\cite{dorfler14,dorfler18,caliskan14}.

\begin{figure}[h]
\begin{center}
\includegraphics[scale=0.45]{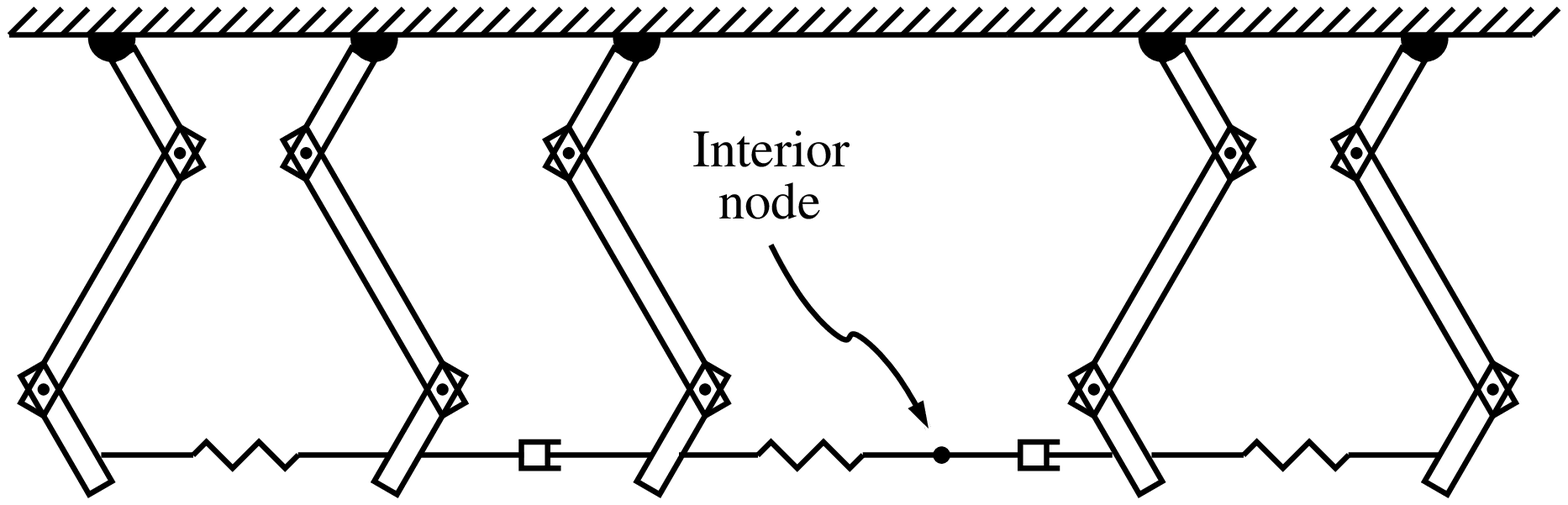}
\caption {Three-link pendulums coupled through a network with an
interior node.}\label{fig:pendula3}
\end{center}
\end{figure}

\section{Problem statement and notation}

Consider the array of $q$ identical oscillators
\begin{eqnarray}\label{eqn:oscillator}
M{\ddot x}_{i}+Kx_{i}=Bu_{i}\,,\quad y_{i}=B^{T}x_{i}\,,\qquad
i=1,\,2,\,\ldots,\,q\,;
\end{eqnarray}
where $x_{i}\in\Real^{n}$; $u_{i},\,y_{i}\in\Real$, the matrices
$M,\,K\in\Real^{n\times n}$ are symmetric positive definite, and
$B\in\Real^{n\times 1}$. We assume that the dynamics $M{\ddot
x}_{i}+Kx_{i}=Bu_{i}$ are controllable through $u_{i}$. Note that
this is equivalent to assuming observability from
$y_{i}=B^{T}x_{i}$. Or, more formally,
\begin{eqnarray}\label{eqn:obs}
{\rm
rank}\left[\begin{array}{c}K-\omega^{2}M\\B^{T}\end{array}\right]=n\quad\mbox{for
all}\quad \omega\in\Real_{>0}\,.
\end{eqnarray}
We study the setup where these oscillators are coupled through a
network with $p\geq q$ nodes, subject to the following constraints
\begin{subeqnarray}\label{eqn:coupling}
u_{i}+\sum_{j=1}^{p}d_{ij}({\dot z}_{i}-{\dot
z}_{j})+\sum_{j=1}^{p}r_{ij}(z_{i}-z_{j})&=&0\,,\qquad
i=1,\,2,\,\ldots,\,q\,;\\
\sum_{j=1}^{p}d_{ij}({\dot z}_{i}-{\dot
z}_{j})+\sum_{j=1}^{p}r_{ij}(z_{i}-z_{j})&=&0\,, \qquad
i=q+1,q+2,\,\,\ldots,\,p\,;\\
z_{i}-y_{i}&=&0\,, \qquad i=1,\,2,\,\ldots,\,q\,;
\end{subeqnarray}
where the indices $i=1,\,2,\,\ldots,\,q$ correspond to the
non-interior nodes (sometimes called the {\em boundary} nodes) and
$i=q+1,q+2,\,\,\ldots,\,p$ to the interior nodes. Also,
$z_{i}\in\Real$, the scalars $d_{ij}=d_{ji}\geq 0$ represent the
dissipative coupling, and the scalars $r_{ij}=r_{ji}\geq 0$
represent the restorative coupling. (We take $d_{ii}=0$ and
$r_{ii}=0$.) In this paper we tackle and propose a solution to the
following problem:

\vspace{0.12in}

\noindent{\bf Problem.} Find conditions on the set of parameters
$(M,\,K,\,B,\,(d_{ij})_{i,j=1}^{p},\,(r_{ij})_{i,j=1}^{p})$ under
which the array~\eqref{eqn:oscillator} coupled through the
constraints~\eqref{eqn:coupling} {\em synchronizes}, i.e.,
$\|x_{i}(t)-x_{j}(t)\|\to 0$ as $t\to\infty$ for all indices
$i,\,j$.

\vspace{0.12in}

Further notation. The identity matrix is denoted by
$I_{q}\in\Real^{q\times q}$. We let $\one_{q}\in\Real^{q}$ denote
the vector of all ones. Given $X\in\Complex^{q\times q}$, we let
$\lambda_{k}(X)$ denote the $k$th smallest eigenvalue of $X$ with
respect to the real part. That is, ${\rm Re}\,\lambda_{1}(X)\leq{\rm
Re}\,\lambda_{2}(X)\leq\cdots\leq{\rm Re}\,\lambda_{q}(X)$. All the
positive (semi)definite matrices we consider in this paper will be
(real and) symmetric. Therefore henceforth we write $X>0$ ($X\geq
0$) to mean $X^{T}=X>0$ ($X^{T}=X\geq 0$). A simple fact from linear
algebra that we frequently use in our analysis is
\begin{eqnarray*}
X\geq 0\ \ \mbox{and}\ \ \xi^{*}X\xi=0\implies X\xi=0
\end{eqnarray*}
where $\xi$ is a vector of appropriate size and $\xi^{*}$ is its
conjugate transpose. Given a set of real (scalar) weights
$(w_{ij})_{i,j=1}^{q}$ with $w_{ij}=w_{ji}\geq 0$ and $w_{ii}=0$,
the associated Laplacian matrix is defined as
\begin{eqnarray*}
L =
\left[\begin{array}{cccc}\sum_{j}w_{1j}&-w_{12}&\cdots&-w_{1q}\\
-w_{21}&\sum_{j}w_{2j}&\cdots&-w_{2q}\\
\vdots&\vdots&\ddots&\vdots\\
-w_{q1}&-w_{q2}&\cdots&\sum_{j}w_{qj}
\end{array}\right]=:{\rm lap}\,(w_{ij})_{i,j=1}^{q}\,.
\end{eqnarray*}
Observe the symmetry $L=L^{T}$ and the positive semidefiniteness
$\xi^{*}L\xi=\sum_{j>i}w_{ij}|\xi_{i}-\xi_{j}|^{2}\geq 0$, where
$\xi=[\xi_{1}\ \xi_{2}\ \cdots\ \xi_{q}]^{T}\in\Complex^{q}$. Also,
we have $L\one_{q}=0$.

\section{Schur complement}

Consider the network~\eqref{eqn:coupling}. Note that the dissipative
coupling $(d_{ij})_{i,j=1}^{p}$ and the restorative coupling
$(r_{ij})_{i,j=1}^{p}$ can be represented, respectively, by the
following pair of $p\times p$ Laplacians
\begin{eqnarray}\label{eqn:DandR}
D={\rm lap}\,(d_{ij})_{i,j=1}^{p}\quad \mbox{and}\quad R={\rm
lap}\,(r_{ij})_{i,j=1}^{p}\,.
\end{eqnarray}
As we mentioned earlier, when there are no interior nodes, i.e.,
when $p=q$, one has to check the number of eigenvalues of the matrix
$[D+jR]$ on the imaginary axis to determine whether the
array~\eqref{eqn:oscillator} coupled through the
constraints~\eqref{eqn:coupling} synchronizes or not. (This is yet
to be shown.) If there are interior nodes, however, i.e., if $p>q$,
then this eigenvalue test should be applied to the Schur complement
of $[D+jR]$. Playing such an essential role in our solution
approach, it is worthwhile that we provide the formal definition of
Schur complement and introduce some of its properties relevant for
the analysis to follow.

\begin{definition}
Given matrices $X\in\Complex^{p\times p}$ and $Y\in\Complex^{q\times
q}$ with $p\geq q$, the matrix $Y$ is said to be a {\em Schur
complement} of $X$ if it satisfies
\begin{eqnarray}\label{eqn:schur}
X\left[\begin{array}{c}I_{q}\\
E\end{array}\right]=\left[\begin{array}{c}Y\\
0\end{array}\right]
\end{eqnarray}
for some $E\in\Complex^{(p-q)\times q}$. When \eqref{eqn:schur}
holds we write $Y={\rm schur}\,(X,\,q)$.
\end{definition}

The following Schur complements will be of key importance for our
analysis:
\begin{eqnarray*}
\Gamma&:=&{\rm schur}\,(D+jR,\,q)\,,\label{eqn:Gamma}\\
\Lambda&:=&{\rm schur}\,(R,\,q)\,.\label{eqn:Lambda}
\end{eqnarray*}
We now investigate certain properties of these matrices.

\begin{lemma}\label{lem:Gamma}
The matrix $\Gamma\in\Complex^{q\times q}$ uniquely exists.
Moreover, it satisfies
\begin{enumerate}
\item $\Gamma^{T}=\Gamma$,
\item ${\rm Re}\,\lambda_{1}(\Gamma)\geq 0$ and ${\rm Im}\,\lambda_{k}(\Gamma)\geq
0$ for all $k$,
\item $\Gamma \one_{q}=0$.
\end{enumerate}
\end{lemma}

\begin{proof}
{\em Existence.} Let us partition the matrices $D$ and $R$ as
\begin{eqnarray}\label{eqn:blocks}
D=\left[\begin{array}{cc}U_{1}&V_{1}\\V_{1}^{T}&W_{1}\end{array}\right]\quad\mbox{and}\quad
R=\left[\begin{array}{cc}U_{2}&V_{2}\\V_{2}^{T}&W_{2}\end{array}\right]
\end{eqnarray}
where $U_{k}\in\Real^{q\times q}$, $V_{k}\in\Real^{q\times (p-q)}$,
and $W_{k}\in\Real^{(p-q)\times (p-q)}$ for $k=1,\,2$. Note that
$D,\,R\geq 0$ implies $W_{k}\geq 0$. By definition we have
\begin{eqnarray}\label{eqn:gamdef}
\left[\begin{array}{c}\Gamma
\\0\end{array}\right]&=&[D+jR]\left[\begin{array}{c} I_{q}
\\E\end{array}\right]\\
&=&\left[\begin{array}{cc}U_{1}+jU_{2}&V_{1}+jV_{2}\\V_{1}^{T}+jV_{2}^{T}&W_{1}+jW_{2}\end{array}\right]\left[\begin{array}{c} I_{q}
\\E\end{array}\right]\,.\nonumber
\end{eqnarray}
Hence $\Gamma$ exists if the equation
\begin{eqnarray*}
[V_{1}^{T}+jV_{2}^{T}]+[W_{1}+jW_{2}]E=0
\end{eqnarray*}
admits a solution $E\in\Complex^{(p-q)\times q}$, which is possible
when ${\rm range}\,[W_{1}+jW_{2}]\supset{\rm
range}\,[V_{1}^{T}+jV_{2}^{T}]$ or, equivalently,
\begin{eqnarray}\label{eqn:null}
{\rm null}\,[W_{1}-jW_{2}]\subset{\rm null}\,[V_{1}-jV_{2}]\,.
\end{eqnarray}
To establish the relation~\eqref{eqn:null} suppose otherwise. Then
there would exist $\eta\in\Complex^{p-q}$ satisfying
$[W_{1}-jW_{2}]\eta=0$ while $[V_{1}-jV_{2}]\eta\neq 0$. We can
write
$0=\eta^{*}[W_{1}-jW_{2}]\eta=\eta^{*}W_{1}\eta-j\eta^{*}W_{2}\eta$.
This allows us to see $\eta^{*}W_{k}\eta=0$ (in fact $W_{k}\eta=0$)
since $W_{k}\geq 0$. Now, let us construct the vector
$\zeta=[0_{1\times q}\ \eta^{T}]^{T}\in\Complex^{p}$. Using $\zeta$
we obtain
\begin{eqnarray*}
\zeta^{*}D\zeta+j\zeta^{*}R\zeta=[0\ \
\eta^{*}]\left[\begin{array}{cc}U_{1}+jU_{2}&V_{1}+jV_{2}\\V_{1}^{T}+jV_{2}^{T}&W_{1}+jW_{2}\end{array}\right]\left[\begin{array}{c}
0
\\\eta\end{array}\right]
=\eta^{*}W_{1}\eta+j\eta^{*}W_{2}\eta =0\,.
\end{eqnarray*}
Since $D,\,R\geq 0$ this implies $D\zeta=0$ and $R\zeta=0$. By
writing
\begin{eqnarray*}
0=D\zeta=\left[\begin{array}{cc}U_{1}&V_{1}\\V_{1}^{T}&W_{1}\end{array}\right]\left[\begin{array}{c}
0
\\\eta\end{array}\right]=\left[\begin{array}{c}
V_{1}\eta
\\W_{1}\eta\end{array}\right]
\end{eqnarray*}
we see that $V_{1}\eta=0$. Similarly, $R\zeta=0$ yields
$V_{2}\eta=0$. But this means $[V_{1}-jV_{2}]\eta=0$.

{\em Symmetry.} By \eqref{eqn:gamdef} we can write
\begin{eqnarray*}
\Gamma = [I_{q}\ \ E^{T}]\left[\begin{array}{c}\Gamma
\\0\end{array}\right]=[I_{q}\ \ E^{T}][D+jR]\left[\begin{array}{c} I_{q}
\\E\end{array}\right]
\end{eqnarray*}
which makes it clear that $\Gamma^{T}=\Gamma$ since $[D+jR]$ is
symmetric.

{\em Uniqueness.} Suppose not. Then we could find
$E_{1},\,E_{2}\in\Complex^{(p-q)\times q}$ satisfying
\begin{eqnarray*}
\left[\begin{array}{c}\Gamma_{1}
\\0\end{array}\right]=[D+jR]\left[\begin{array}{c} I_{q}
\\E_{1}\end{array}\right]\quad\mbox{and}\quad \left[\begin{array}{c}\Gamma_{2}
\\0\end{array}\right]=[D+jR]\left[\begin{array}{c} I_{q}
\\E_{2}\end{array}\right]
\end{eqnarray*}
and yielding $\Gamma_{1}\neq\Gamma_{2}$. But, since
$[D+jR]=[D+jR]^{T}$, this leads to the contradiction:
\begin{eqnarray*}
\Gamma_{1}=[I_{q}\ \ E_{2}^{T}]\left[\begin{array}{c}\Gamma_{1}
\\0\end{array}\right]=[I_{q}\ \ E_{2}^{T}][D+jR]\left[\begin{array}{c} I_{q}
\\E_{1}\end{array}\right]=[\Gamma_{2}^{T}\ \ 0]\left[\begin{array}{c} I_{q}
\\E_{1}\end{array}\right]=\Gamma_{2}^{T}=\Gamma_{2}\,.
\end{eqnarray*}

{\em That ${\rm Re}\,\lambda_{1}(\Gamma)\geq 0$ and ${\rm
Im}\,\lambda_{k}(\Gamma)\geq 0$ for all $k$.} Let
$\lambda\in\Complex$ be an eigenvalue of $\Gamma$ with eigenvector
$v\in\Complex^q$. Without loss of generality suppose $v^{*}v=1$. Let
us construct the vector $\zeta = [v^{T}\
(Ev)^{T}]^{T}\in\Complex^{p}$ where $E$ satisfies
\eqref{eqn:gamdef}. We can write
\begin{eqnarray*}
\lambda=[v^{*}\ \ (Ev)^{*}]\left[\begin{array}{c} \lambda v
\\ 0\end{array}\right]
=\zeta^{*}\left[\begin{array}{c} \Gamma
\\ 0\end{array}\right]v=\zeta^{*}[D+jR]\left[\begin{array}{c} I_{q}
\\E\end{array}\right]v=\zeta^{*}[D+jR]\zeta=\zeta^{*}D\zeta+j\zeta^{*}R\zeta\,.
\end{eqnarray*}
Since $D,\,R\geq 0$ we have ${\rm Re}\,\lambda=\zeta^{*}D\zeta\geq
0$ and ${\rm Im}\,\lambda=\zeta^{*}R\zeta\geq 0$.

{\em That $\Gamma \one_{q}=0$.} Recall that $D\one_{p}=R\one_{p}=0$.
By the symmetry of the matrices $D,\,R,\,\Gamma$ and using
\eqref{eqn:gamdef} we can write
\begin{eqnarray*}
\Gamma\one_{q}=\Gamma^{T}\one_{q}=[\Gamma^{T}\ \ 0]\one_{p}=[I_{q}\
\ E^{T}][D+jR]^{T}\one_{p}=[I_{q}\ \ E^{T}][D\one_{p}+jR\one_{p}]=0
\end{eqnarray*}
which was to be shown.
\end{proof}

\begin{lemma}\label{lem:Lambda}
The matrix $\Lambda\in\Real^{q\times q}$ uniquely exists. Moreover,
it is symmetric positive semidefinite and satisfies
 $\Lambda\one_{q}=0$.
\end{lemma}

\begin{proof}
Existence, uniqueness, and that $\Lambda\one_{q}=0$ can be
established through the steps taken in the proof of
Lemma~\ref{lem:Gamma}. To show that $\Lambda$ is real, symmetric
positive semidefinite we note that it satisfies
\begin{eqnarray}\label{eqn:lamdef}
\left[\begin{array}{c}\Lambda
\\0\end{array}\right]=R\left[\begin{array}{c} I_{q}
\\E\end{array}\right]
\end{eqnarray}
for some $E\in\Complex^{(p-q)\times q}$. This allows us to write
\begin{eqnarray*}
\Lambda=[I_{q}\ \ E^{*}]\left[\begin{array}{c}\Lambda
\\0\end{array}\right]=[I_{q}\ \ E^{*}]R\left[\begin{array}{c} I_{q}
\\E\end{array}\right]\,.
\end{eqnarray*}
Hence $\Lambda\geq 0$ because $R\geq 0$.
\end{proof}

\begin{remark}
The matrices $\Gamma$ and $\Lambda$ can be computed through
\begin{eqnarray*}
\Gamma&=&[U_{1}+jU_{2}]-[V_{1}+jV_{2}][W_{1}+jW_{2}]^{+}[V_{1}+jV_{2}]^{T}\\
\Lambda&=&U_{2}-V_{2}W_{2}^{+}V_{2}^{T}
\end{eqnarray*}
where the blocks $U_{k},\,V_{k},\,W_{k}$ are as defined in
\eqref{eqn:blocks} and $[\,\cdot\,]^{+}$ denotes the pseudoinverse.
\end{remark}

Now we are ready to present the main result of this paper.

\begin{theorem}\label{thm:main}
The array~\eqref{eqn:oscillator} coupled through the
constraints~\eqref{eqn:coupling} synchronizes if and only if ${\rm
Re}\,\lambda_{2}(\Gamma)>0$.
\end{theorem}

In the remainder of the paper we establish Theorem~\ref{thm:main}.
We do this in two steps, each expounded in a separate section.
First, we discover the constraints (that are more stringent than
\eqref{eqn:coupling}) to which the steady-state solutions of the
array~\eqref{eqn:oscillator} are subject. And from those constraints
we obtain conditions for synchronization in the form of algebraic
equations (as opposed to the original differential equations).
Second, we make a connection between those algebraic constraints and
the second eigenvalue of the matrix $\Gamma$ and therefore prove the
main result.

\section{Steady-state solutions}

Imagine the coupled pendulums in Fig.~\ref{fig:pendula3} in motion.
Since there is no external source of energy input to this system,
the total mechanical energy trapped in it can never increase. On the
contrary, due to the presence of dampers, some of this energy should
gradually leak out of the assembly as heat. This means that in the
long run the overall system should settle into a constant energy
state, the {\em steady state}. Now, it is not difficult to see that
if in the steady state the oscillators are not in synchrony then
they never will be. This simple observation is what guides us in the
analysis to follow.

Consider the array~\eqref{eqn:oscillator} under the
coupling~\eqref{eqn:coupling}. For our analysis let us construct the
aggregate vectors
\begin{eqnarray*}
x=\left[\begin{array}{c}x_{1}\\ x_{2}\\ \vdots\\
x_{q}\end{array}\right]\,,\quad u=\left[\begin{array}{c}u_{1}\\ u_{2}\\ \vdots\\
u_{q}\end{array}\right]\,,\quad y=\left[\begin{array}{c}y_{1}\\ y_{2}\\ \vdots\\
y_{q}\end{array}\right]\,,\quad z=\left[\begin{array}{c}z_{1}\\ z_{2}\\ \vdots\\
z_{p}\end{array}\right]\,,\quad g=\left[\begin{array}{c}z_{q+1}\\ z_{q+2}\\ \vdots\\
z_{p}\end{array}\right]\,.
\end{eqnarray*}
Note that $z=[y^{T}\ g^{T}]^{T}$ and $y=[I_{q}\otimes B^{T}]x$ where
$\otimes$ denotes the Kronecker product. Let $\|x\|_{\rm s}$ denote
the {\em distance of $x$ to the synchronization subspace}. Namely,
$\|x\|_{\rm s}=\|[(I_{q}-q^{-1}\one_{q}\one_{q}^{T})\otimes
I_{n}]x\|$, where $[I_{q}-q^{-1}\one_{q}\one_{q}^{T}]$ is no other
than the orthogonal projection matrix that projects onto $({\rm
range}\,\one_{q})^{\perp}={\rm null}\,\one_{q}^{T}$. Note that
$\|x\|_{\rm s}=0$ means $\|x_{i}-x_{j}\|=0$ for all $i,\,j$.

Let us now compactly reexpress the coupled
dynamics~\eqref{eqn:oscillator} and \eqref{eqn:coupling} as
\begin{subeqnarray}\label{eqn:array}
[I_{q}\otimes M]{\ddot x}+[I_{q}\otimes K]x&=&[I_{q}\otimes B]u\,,\quad y = [I_{q}\otimes B^{T}]x\\[0.5ex]
\left[\begin{array}{c}u\\
0\end{array}\right]+D\underbrace{\left[\begin{array}{c}{\dot
y}\\{\dot g}
\end{array}\right]}_{\displaystyle \dot z}+R\underbrace{\left[\begin{array}{c}y \\ g
\end{array}\right]}_{\displaystyle z}&=&0
\end{subeqnarray}
where the Laplacian matrices $D,\,R$ are defined in
\eqref{eqn:DandR}. To gain some understanding on the qualitative
behavior of the trajectories resulting from the dynamics
\eqref{eqn:array}, construct the following nonnegative function
\begin{eqnarray*}
W(t)=\frac{1}{2}x(t)^{T}[I_{q}\otimes K]x(t)+\frac{1}{2}{\dot
x(t)}^{T}[I_{q}\otimes M]{\dot x(t)}+\frac{1}{2}z(t)^{T}Rz(t)\,.
\end{eqnarray*}
The time derivative of this function along the solutions of
\eqref{eqn:array} reads
\begin{eqnarray}\label{eqn:Dzdot}
\frac{d}{dt}W(t)=-{\dot z}(t)^{T}D{\dot z}(t)
\end{eqnarray}
which allows us to claim ${\dot W}\leq 0$ because $D$ is symmetric
positive semidefinite. Being nonnegative, $W$ is lower bounded. From
this the convergence ${\dot W}\to 0$ follows. Then, since $D\geq 0$,
by \eqref{eqn:Dzdot} we can write
\begin{eqnarray}\label{eqn:limit}
\lim_{t\to\infty} D{\dot z}(t)=0\,.
\end{eqnarray}
The existence of the limit~\eqref{eqn:limit} implies that every
trajectory $x(t)$ evolving in accordance with \eqref{eqn:array} must
converge to a {\em steady-state} trajectory $x_{\rm ss}(t)$ that is
subject to
\begin{subeqnarray}\label{eqn:arrayss}
[I_{q}\otimes M]{\ddot x}+[I_{q}\otimes K]x&=&[I_{q}\otimes B]u\,,\quad y = [I_{q}\otimes B^{T}]x\\[0.5ex]
\left[\begin{array}{c}u\\
0\end{array}\right]+R\left[\begin{array}{c}y \\ g
\end{array}\right]&=&0\\[0.5ex]
D\left[\begin{array}{c}{\dot y}\\{\dot g}
\end{array}\right]&=&0
\end{subeqnarray}
for some $g(t)$. Moreover, it is clear that every $(x(t),\,g(t))$
pair satisfying \eqref{eqn:arrayss} automatically satisfies
\eqref{eqn:array}. Hence we have established:

\begin{lemma}\label{lem:ss}
The array~\eqref{eqn:oscillator} coupled through the
constraints~\eqref{eqn:coupling} synchronizes if and only if the
array~\eqref{eqn:arrayss} synchronizes, i.e., all solutions satisfy
$\|x(t)\|_{\rm s}\to 0$ as $t\to\infty$.
\end{lemma}

Let us now focus on the dynamics~\eqref{eqn:arrayss}. Recall $R\geq
0$ and  $\Lambda={\rm schur}\,(R,\,q)\geq 0$, the latter by
Lemma~\ref{lem:Lambda}. Using $R^{T}=R$, $\Lambda^{T}=\Lambda$, and
(\ref{eqn:arrayss}b) we can write
\begin{eqnarray}\label{eqn:baru}
u=[I_{q}\ \ E^{T}]\left[\begin{array}{c}u\\
0\end{array}\right]=-[I_{q}\ \ E^{T}]R\left[\begin{array}{c}y\\
g\end{array}\right]=-[\Lambda\ \ 0]\left[\begin{array}{c}y\\
g\end{array}\right]=-\Lambda y
\end{eqnarray}
where $E\in\Complex^{(p-q)\times q}$ satisfies \eqref{eqn:lamdef}.
This allows us to proceed as
\begin{eqnarray*}
[I_{q}\otimes B]u =-[I_{q}\otimes B]\Lambda y =-[I_{q}\otimes
B]\Lambda [I_{q}\otimes B^{T}]x =-[\Lambda\otimes BB^{T}]x\,.
\end{eqnarray*}
Then (\ref{eqn:arrayss}a) yields
\begin{eqnarray}\label{eqn:arnold}
[I_{q}\otimes M]{\ddot x}+\left([I_{q}\otimes K]+[\Lambda\otimes
BB^{T}]\right)x=0\,.
\end{eqnarray}
Since $M,\,K>0$ and $\Lambda,\,BB^{T}\geq 0$ the matrices
$[I_{q}\otimes M]$ and $\left([I_{q}\otimes K]+[\Lambda\otimes
BB^{T}]\right)$ are both symmetric positive definite. Hence the
solution to \eqref{eqn:arnold} has the form \cite[\S 23]{arnold89}
\begin{eqnarray}\label{eqn:form}
x(t)={\rm Re}\,\sum_{k=1}^{N} e^{j\omega_{k}t}\xi_{k}
\end{eqnarray}
where $N\leq qn$, the positive scalars $\omega_{k}$ are distinct,
and each $\xi_{k}\in(\Complex^n)^{q}$ satisfies
\begin{eqnarray*}
\left([I_{q}\otimes (K-\omega_{k}^{2}M)]+[\Lambda\otimes
BB^{T}]\right)\xi_{k}=0\,.
\end{eqnarray*}
Now, suppose that the array~\eqref{eqn:arrayss} fails to
synchronize. This implies that there exists a pair $(x(t),\,g(t))$
satisfying \eqref{eqn:arrayss}, where $x(t)$ also satisfies
\eqref{eqn:form} with $\xi_{k}\notin{\rm range}\,[\one_{q}\otimes
I_{n}]$ for some $k$. Using this particular index $k$ let us define
${\bar x}=\xi_{k}$ and $\omega=\omega_{k}$. Since the dynamics are
LTI, there should also exist a pair $({\hat x}(t),\,{\hat g}(t))$
satisfying \eqref{eqn:arrayss} such that ${\hat x}(t) = {\rm
Re}\,(e^{j\omega t}{\bar x})$ and ${\hat g}(t) = {\rm
Re}\,(e^{j\omega t}{\bar g})$ for some ${\bar g}\in\Complex^{p-q}$.
A straightforward implication of this on the triple $(\omega,\,{\bar
x},\,{\bar g})$ is then the following set of constraints
\begin{subeqnarray}\label{eqn:arraystatic}
[I_{q}\otimes(K-\omega^{2}M)]{\bar x}&=&[I_{q}\otimes B]{\bar u}\,,\quad {\bar y} = [I_{q}\otimes B^{T}]{\bar x}\\[0.5ex]
\left[\begin{array}{c}{\bar u}\\
0\end{array}\right]+R\left[\begin{array}{c}{\bar y} \\ {\bar g}
\end{array}\right]&=&0\\[0.5ex]
D\left[\begin{array}{c}{\bar y}\\{\bar g}
\end{array}\right]&=&0
\end{subeqnarray}
where ${\bar u}=-\Lambda{\bar y}$ by \eqref{eqn:baru}.

Conversely, if we can find $\omega\in\Real_{>0}$, ${\bar
x}\in(\Complex^{n})^{q}$ with ${\bar x}\notin{\rm
range}\,[\one_{q}\otimes I_{n}]$, and ${\bar g}\in\Complex^{p-q}$
satisfying \eqref{eqn:arraystatic}, we can claim that the
array~\eqref{eqn:arrayss} fails to synchronize. The reason is that
we can construct the functions ${\hat x}(t) = {\rm Re}\,(e^{j\omega
t}{\bar x})$ and ${\hat g}(t) = {\rm Re}\,(e^{j\omega t}{\bar g})$
which satisfy the dynamics~\eqref{eqn:arrayss} thanks to
\eqref{eqn:arraystatic}. And since $\|{\hat x}(t)\|_{\rm s}\not\to
0$ (because ${\bar x}\notin{\rm range}\,[\one_{q}\otimes I_{n}]$)
our claim holds. To summarize:

\begin{lemma}\label{lem:ss2}
The array~\eqref{eqn:arrayss} fails to synchronize if and only if
there exist $\omega\in\Real_{>0}$, ${\bar x}\in(\Complex^{n})^{q}$
with ${\bar x}\notin{\rm range}\,[\one_{q}\otimes I_{n}]$, and
${\bar g}\in\Complex^{p-q}$ satisfying \eqref{eqn:arraystatic}.
\end{lemma}

\section{The second eigenvalue}

In this section we conduct the last part of our analysis, thereby
completing the preparation for the proof of Theorem~\ref{thm:main}.
Namely, we show that the condition on the triple $(\omega,\,{\bar
x},\,{\bar g})$ presented in Lemma~\ref{lem:ss2} is equivalent to a
condition on the second eigenvalue of the matrix $\Gamma={\rm
schur}\,(D+jR,\,q)$. We establish this equivalence in two lemmas;
one for necessity, the other for sufficiency.

\begin{lemma}\label{lem:lam2sync}
If ${\rm Re}\,\lambda_{2}(\Gamma)\leq 0$ then there exist
$\omega\in\Real_{>0}$, ${\bar x}\in(\Complex^{n})^{q}$ with ${\bar
x}\notin{\rm range}\,[\one_{q}\otimes I_{n}]$, and ${\bar
g}\in\Complex^{p-q}$ satisfying \eqref{eqn:arraystatic}.
\end{lemma}

\begin{proof}
Let ${\rm Re}\,\lambda_{2}(\Gamma)\leq 0$. By Lemma~\ref{lem:Gamma}
we have $\Gamma \one_{q}=0$ (i.e., $\Gamma$ has an eigenvalue at the
origin) and ${\rm Re}\,\lambda_{1}(\Gamma)\geq 0$. Hence, without
loss of generality, we can let $\lambda_{1}(\Gamma)=0$. Note that
${\rm Re}\,\lambda_{1}(\Gamma)\geq 0$ implies ${\rm
Re}\,\lambda_{k}(\Gamma)\geq 0$ for all $k$. This yields ${\rm
Re}\,\lambda_{2}(\Gamma)=0$ because ${\rm
Re}\,\lambda_{2}(\Gamma)\leq 0$. Lemma~\ref{lem:Gamma} also tells us
that ${\rm Im}\,\lambda_{2}(\Gamma)\geq 0$. Therefore ${\rm
Re}\,\lambda_{2}(\Gamma)=j\mu$ for some real number $\mu\geq 0$.
Choose now a unit vector ${\bar y}\in\Complex^{q}$ satisfying ${\bar
y}\notin{\rm span}\,\{\one_{q}\}$ and $\Gamma{\bar y}=\mu{\bar y}$.
Such eigenvector exists even if $\mu=0$ (i.e., the eigenvalue at the
origin is repeated) because $\Gamma^{T}=\Gamma$ by
Lemma~\ref{lem:Gamma}. To see that suppose otherwise, i.e.,
$\one_{q}$ were the only eigenvector for the repeated eigenvalue at
the origin. But then there would exist a generalized eigenvector
$v\in\Complex^{q}$ satisfying $\Gamma v=\one_{q}$ and we would have
the following contradiction:
$0=(\Gamma\one_{q})^{T}v=\one_{q}^{T}\Gamma
v=\one_{q}^{T}\one_{q}=q$. Let now $E\in\Complex^{(p-q)\times q}$
satisfy \eqref{eqn:gamdef}. Then define ${\bar g}=E{\bar y}$. By
\eqref{eqn:gamdef} we can write
\begin{eqnarray}\label{eqn:wuhuu}
[D+jR]\left[\begin{array}{c}{\bar y}\\ {\bar
g}\end{array}\right]=[D+jR]\left[\begin{array}{c}I_{q}\\
E\end{array}\right]{\bar y}=\left[\begin{array}{c}\Gamma\\
0\end{array}\right]{\bar y}=\left[\begin{array}{c}\Gamma{\bar y}\\
0\end{array}\right]=j\mu\left[\begin{array}{c}{\bar y}\\
0\end{array}\right]\,.
\end{eqnarray}
Let $\zeta=[{\bar y}^{T}\ \ {\bar g}^{T}]^{T}$. By \eqref{eqn:wuhuu}
we have
\begin{eqnarray*}
\zeta^{*}D\zeta+j\zeta^{*}R\zeta=\zeta^{*}[D+jR]\zeta=j\mu[{\bar
y}^{*}\
\ {\bar g}^{*}]\left[\begin{array}{c}{\bar y}\\
0\end{array}\right]=j\mu
\end{eqnarray*}
thanks to ${\bar y}^{*}{\bar y}=1$. Since $D,\,R\geq0$, both
$\zeta^{*}D\zeta$ and $\zeta^{*}R\zeta$ are real. This means
$\zeta^{*}D\zeta=0$ which in turn implies
\begin{eqnarray}\label{eqn:static3}
D\left[\begin{array}{c}{\bar y}\\{\bar g}
\end{array}\right]=0\,.
\end{eqnarray}
Now, by \eqref{eqn:wuhuu} and \eqref{eqn:static3} we have
\begin{eqnarray*}
R\left[\begin{array}{c}{\bar y}\\ {\bar g}\end{array}\right]=\left[\begin{array}{c}\mu{\bar y}\\
0\end{array}\right]\,.
\end{eqnarray*}
Hence, letting ${\bar u}=-\mu{\bar y}$ we can write
\begin{eqnarray}\label{eqn:static2}
\left[\begin{array}{c}{\bar u}\\
0\end{array}\right]+R\left[\begin{array}{c}{\bar y} \\ {\bar g}
\end{array}\right]=0\,.
\end{eqnarray}
Now, choose some $\omega>0$ for which $[K-\omega^{2}M+\mu BB^{T}]$
is singular. Then let $\eta\in{\rm null}\,[K-\omega^{2}M+\mu
BB^{T}]$ be a nonzero vector satisfying $B^{T}\eta=1$. Such $\eta$
exists (i.e., $\eta\notin{\rm null}\,B^{T}$) thanks to the
observability assumption~\eqref{eqn:obs}. Then construct ${\bar
x}=[{\bar y}\otimes\eta]$, which satisfies ${\bar x}\notin {\rm
range}\,[\one_{q}\otimes I_{n}]$ because ${\bar y}\notin{\rm
span}\,\{\one_{q}\}$. Observe that
\begin{eqnarray}\label{eqn:static1b}
{\bar y}=[{\bar y}\otimes B^{T}\eta]=[I_{q}\otimes B^{T}][{\bar
y}\otimes\eta]=[I_{q}\otimes B^{T}]{\bar x}\,.
\end{eqnarray}
Also, since $[K-\omega^{2}M+\mu BB^{T}]\eta=0$ and ${\bar
u}=-\mu{\bar y}$,
\begin{eqnarray}\label{eqn:static1a}
[I_{q}\otimes(K-\omega^{2}M)]{\bar x}
&=&[I_{q}\otimes(K-\omega^{2}M)][{\bar y}\otimes\eta]\nonumber\\
&=&[{\bar y}\otimes(K-\omega^{2}M)\eta]\nonumber\\
&=&[{\bar y}\otimes(-\mu BB^{T}\eta)]\nonumber\\
&=&[(-\mu{\bar y})\otimes B]\nonumber\\
&=&[{\bar u}\otimes B]\nonumber\\
&=&[I_{q}\otimes B]{\bar u}\,.
\end{eqnarray}
The result follows by \eqref{eqn:static3}, \eqref{eqn:static2},
\eqref{eqn:static1b}, and \eqref{eqn:static1a}.
\end{proof}

\begin{lemma}\label{lem:sync2lam}
If there exist $\omega\in\Real_{>0}$, ${\bar
x}\in(\Complex^{n})^{q}$ with ${\bar x}\notin{\rm
range}\,[\one_{q}\otimes I_{n}]$, and ${\bar g}\in\Complex^{p-q}$
satisfying \eqref{eqn:arraystatic} then ${\rm
Re}\,\lambda_{2}(\Gamma)\leq 0$.
\end{lemma}

\begin{proof}
Let \eqref{eqn:arraystatic} holds for some triple $(\omega,\,{\bar
x},\,{\bar g})$ with ${\bar x}\notin {\rm range}\,[\one_{q}\otimes
I_{n}]$. We consider two cases.

{\em Case~1: $[K-\omega^{2}M]$ is singular.} Let
$\eta\in\Complex^{n}$ satisfy $[K-\omega^{2}M]\eta=0$ and
$B^{T}\eta=1$. Such $\eta$ exists by the observability
assumption~\eqref{eqn:obs}. Moreover, it is unique; for if it were
not then we could find $\tilde\eta\neq\eta$ satisfying
$[K-\omega^{2}M]\tilde\eta=0$ and $B^{T}\tilde\eta=1$ and construct
the nonzero vector $\hat\eta=\tilde\eta-\eta$ satisfying
$[K-\omega^{2}M]\hat\eta=0$ and $B^{T}\hat\eta=0$, which would
contradict \eqref{eqn:obs}. Since $[K-\omega^{2}M]$ real and
symmetric we have $\eta^{*}[K-\omega^{2}M]=0$ and $\eta^{*}B=1$.
Consider now \eqref{eqn:arraystatic}, where ${\bar u}$ must be zero
because $\|{\bar u}\|^{2}=[({\bar u}^{*}{\bar u})\otimes
(\eta^{*}B)]=[{\bar u}\otimes \eta]^{*}[I_{q}\otimes B]{\bar
u}=[{\bar u}\otimes \eta]^{*}[I_{q}\otimes(K-\omega^{2}M)]{\bar
x}=[{\bar u}^{*}\otimes (\eta^{*}(K-\omega^{2}M))]{\bar x}=0$. Hence
(\ref{eqn:arraystatic}a) becomes
\begin{eqnarray*}
[I_{q}\otimes(K-\omega^{2}M)]{\bar x}=0\,,\quad {\bar y} =
[I_{q}\otimes B^{T}]{\bar x}
\end{eqnarray*}
yielding (by the uniqueness of $\eta$ discussed above) the relation
${\bar x}=[{\bar y}\otimes \eta]$. Then, since ${\bar x}\notin {\rm
range}\,[\one_{q}\otimes I_{n}]$, we have to have ${\bar
y}\notin{\rm span}\,\{\one_{q}\}$. Under ${\bar u}=0$ we have by
\eqref{eqn:arraystatic}
\begin{eqnarray*}
D\left[\begin{array}{c}{\bar y}\\{\bar g}
\end{array}\right]=0\qquad\mbox{and}\qquad R\left[\begin{array}{c}{\bar y}\\{\bar g}
\end{array}\right]=0\,.
\end{eqnarray*}
Using this and the symmetry of the matrices $\Gamma,\,D,\,R$ we can
now proceed as
\begin{eqnarray*}
\Gamma{\bar y}=[\Gamma\ \ 0]\left[\begin{array}{c}{\bar y}\\
{\bar g}\end{array}\right]=[I_{q}\ \ E^{T}][D+jR]\left[\begin{array}{c}{\bar y}\\
{\bar g}\end{array}\right]=0
\end{eqnarray*}
where $E\in\Complex^{(p-q)\times q}$ satisfies \eqref{eqn:gamdef}.
Therefore ${\bar y}$ is an eigenvector of $\Gamma$ for the
eigenvalue at the origin. By Lemma~\ref{lem:Gamma} we also have
$\Gamma\one_{q}=0$. Since ${\bar y}\notin{\rm span}\,\{\one_{q}\}$,
this means that $\Gamma$ has at least two eigenvalues at the origin.
Hence $\lambda_{2}(\Gamma)\leq 0$.

{\em Case~2: $[K-\omega^{2}M]$ is nonsingular.} By \eqref{eqn:baru}
we have ${\bar u}=-\Lambda{\bar y}$. Then (\ref{eqn:arraystatic}a)
yields
\begin{eqnarray*}
\left([I_{q}\otimes(K-\omega^{2}M)]+[\Lambda\otimes
(BB^{T})]\right){\bar x}=0\,.
\end{eqnarray*}
Pre-multiplying this equation by $[I_{q}\otimes
(B^{T}[K-\omega^{2}M]^{-1})]$ and defining
$\alpha=B^{T}[K-\omega^{2}M]^{-1}B$ we obtain
\begin{eqnarray}\label{eqn:premultiply}
0&=&[I_{q}\otimes
(B^{T}[K-\omega^{2}M]^{-1})]\left([I_{q}\otimes(K-\omega^{2}M)]+[\Lambda\otimes
(BB^{T})]\right){\bar x}\nonumber\\
&=&\left([I_{q}\otimes B^{T}]+[\Lambda\otimes
(B^{T}[K-\omega^{2}M]^{-1}BB^{T})]\right){\bar x}\nonumber\\
&=&\left(I_{q}+[\Lambda\otimes
(B^{T}[K-\omega^{2}M]^{-1}B)]\right)[I_{q}\otimes B^{T}]{\bar x}\nonumber\\
&=&[I_{q}+\alpha\Lambda]{\bar y}\,.
\end{eqnarray}
Note that ${\bar y}$ cannot be zero; for otherwise we would have
${\bar u}=0$ because ${\bar u}=-\Lambda{\bar y}$. Then ${\bar u}=0$
and (\ref{eqn:arraystatic}a) would yield
$[I_{q}\otimes(K-\omega^{2}M)]{\bar x}=0$, which would require
$[K-\omega^{2}M]$ to be singular since ${\bar x}\neq 0$. Hence
${\bar y}\neq 0$ because $[K-\omega^{2}M]^{-1}$ exists. Since ${\bar
y}$ is nonzero, so is $\alpha$ by \eqref{eqn:premultiply}. Let
$\mu=-\alpha^{-1}$. Then \eqref{eqn:premultiply} yields
$\Lambda{\bar y}=\mu{\bar y}$, from which follows ${\bar
u}=-\Lambda{\bar y}=-\mu{\bar y}$ and we have by
\eqref{eqn:arraystatic}
\begin{eqnarray*}
D\left[\begin{array}{c}{\bar y}\\{\bar g}
\end{array}\right]=0\qquad\mbox{and}\qquad R\left[\begin{array}{c}{\bar y}\\{\bar g}
\end{array}\right]=\left[\begin{array}{c}{\mu\bar y}\\ 0
\end{array}\right]\,.
\end{eqnarray*}
Using this and the symmetry of the matrices $\Gamma,\,D,\,R$ we can
write
\begin{eqnarray*}
\Gamma{\bar y}=[\Gamma\ \ 0]\left[\begin{array}{c}{\bar y}\\
{\bar g}\end{array}\right]=[I_{q}\ \ E^{T}][D+jR]\left[\begin{array}{c}{\bar y}\\
{\bar g}\end{array}\right]=[I_{q}\ \ E^{T}]\left[\begin{array}{c}j\mu{\bar y}\\
0\end{array}\right]=j\mu{\bar y}\,.
\end{eqnarray*}
Therefore ${\bar y}$ is an eigenvector of $\Gamma$ for the
eigenvalue $j\mu$. Recall that $\mu$ is real and nonzero.  Also,
$\Gamma\one_{q}=0$. This means that $\Gamma$ has at least two
eigenvalues on the imaginary axis. Hence $\lambda_{2}(\Gamma)\leq
0$.
\end{proof}

\vspace{0.12in}

\noindent{\bf Proof of Theorem~\ref{thm:main}.} Combine
Lemma~\ref{lem:ss}, Lemma~\ref{lem:ss2}, Lemma~\ref{lem:lam2sync},
and Lemma~\ref{lem:sync2lam}. \hfill\null\hfill$\blacksquare$

\section{Conclusion}

In this paper we studied the problem of synchronization in an array
of identical linear oscillators (of arbitrary order) coupled through
a dynamic network with interior nodes. We showed that whether the
oscillators asymptotically synchronize is independent of the
oscillator parameters, and instead depends solely on the coupling
network represented by a complex-valued Laplacian matrix. Our
investigations on this coupling matrix revealed that the oscillators
asymptotically synchronize if and only if the Schur complement of
the Laplacian has a single eigenvalue on the imaginary axis, thus
generalizing a recent theorem on the special case of harmonic
oscillators.

\bibliographystyle{plain}
\bibliography{references}
\end{document}